\documentclass[10pt]{article}

\usepackage{amsthm,amsmath,amssymb,amsbsy,amsfonts, enumitem,graphicx,floatrow}

\usepackage[usenames,dvipsnames,svgnames,table]{xcolor}

\usepackage[ocgcolorlinks,bookmarks=true]{hyperref}

\hypersetup{plainpages = false,
            pdfpagelayout = OneColumn,
            breaklinks = true,
            colorlinks = true,
            linkcolor = NavyBlue,
            urlcolor  = blue,
            citecolor = Brown,
            anchorcolor = green}

\usepackage[hmargin=1.5in,vmargin=1.5in]{geometry}

\makeatletter
\def\blfootnote{\gdef\@thefnmark{}\@footnotetext}
\makeatother

\newtheorem{theorema}{Theorem}

\newtheorem*{conjecturea}{Saff-Varga Width Conjecture}
\newtheorem*{conjectureb}{Modified Width Conjecture}

\newtheorem{theorem}{Theorem}[section]
\newtheorem{lemma}[theorem]{Lemma}

\theoremstyle{definition}
\newtheorem{definition}[theorem]{Definition}

\newtheorem*{acknow}{Acknowledgements}

\numberwithin{equation}{section}
\def\C{\mathbb{C}}
\def\R{\mathbb{R}}

\DeclareMathOperator{\erfc}{erfc}
\DeclareMathOperator{\re}{Re}
\DeclareMathOperator{\im}{Im}

\DeclareMathOperator{\len}{length}

\newcommand{\const}{\text{const}}

\newtheorem{rhp}[theorem]{Riemann-Hilbert Problem}
\newtheorem{condition}[theorem]{Condition}

\begin{document}

\graphicspath{{pdfimgs/}}

%\title[Newman-Rivlin asymptotics]{Newman-Rivlin asymptotics \\ for partial sums of power series}
\title{Newman-Rivlin asymptotics \\ for partial sums of power series}
\author{Antonio R. Vargas}
%\address{Dept.\ of Mathematics and Statistics, Dalhousie University, Halifax, Nova Scotia B3H 4J5, Canada}
%\email{antoniov@mathstat.dal.ca}
\date{}

\blfootnote{This research was supported in part by a Killam Predoctoral Scholarship.}

\maketitle

\begin{abstract}
We discuss analogues of Newman and Rivlin's formula concerning the ratio of a partial sum of a power series to its limit function and present a new general result of this type for entire functions with a certain asymptotic character.  The main tool used in the proof is a Riemann-Hilbert formulation for the partial sums introduced by Kriecherbauer et al.  This new result makes some progress on verifying a part of the Saff-Varga Width Conjecture concerning the zero-free regions of these partial sums.
\end{abstract}

\section{Introduction}
Let $f$ be an entire function and let
\[
	p_n(z) = \sum_{k=0}^{n} \frac{f^{(k)}(0)}{k!} z^k
\]
denote the $n^\text{th}$ partial sum of its power series.  A problem of some interest is to determine the asymptotic behavior of the zeros of these partial sums as $n \to \infty$.

An early example of a result of this type is the paper \cite{szego:exp} by Szeg\H{o} which considered the zeros of the partial sums of the function $f(z) = \exp(z)$.  Among other things, Szeg\H{o} showed that the zeros of the re-scaled partial sums $p_n(nz)$ approach a specific limit curve defined by $|ze^{1-z}| = 1$ with $\re z \leq 1$.  This curve has come to be known as the Szeg\H{o} curve.  A series of other results followed in this vein, including the very general contributions of Rosenbloom in his thesis \cite{rosen:thesis,rosen:distrib}.

At the suggestion of Varga, Iverson published a paper \cite{iverson:zeros} containing tables of numerical values and a plot of the zeros of various partial sums of $f(z) = \exp(z)$.  Iverson remarked that there seemed to be a large zero-free region surrounding the positive real axis which was not yet described by the available literature.  This zero-free region was subsequently investigated by Newman and Rivlin in \cite{newriv:expzeros,newriv:expzeroscorrect} and in a more general setting by Saff and Varga in \cite{sv:zerofree}.  In the latter it was shown that no partial sum has a zero in the parabolic region
\[
	\{x+iy : y^2 \leq 4(x+1) \text{ and } x > -1\}.
\]
Conversely, the first paper by Newman and Rivlin contained the following theorem.

\begin{figure}[!hb]
\floatbox[{\capbeside\thisfloatsetup{capbesideposition={right,center}, capbesidewidth=0.3\textwidth}}]{figure}[\FBwidth]
{\caption{Zeros of the first $50$ partial sums of $\exp(z)$.  The curve $y^2 = 4(x+1)$ is shown as a solid line.}\label{ezplot}}
{\includegraphics[width=0.65\textwidth]{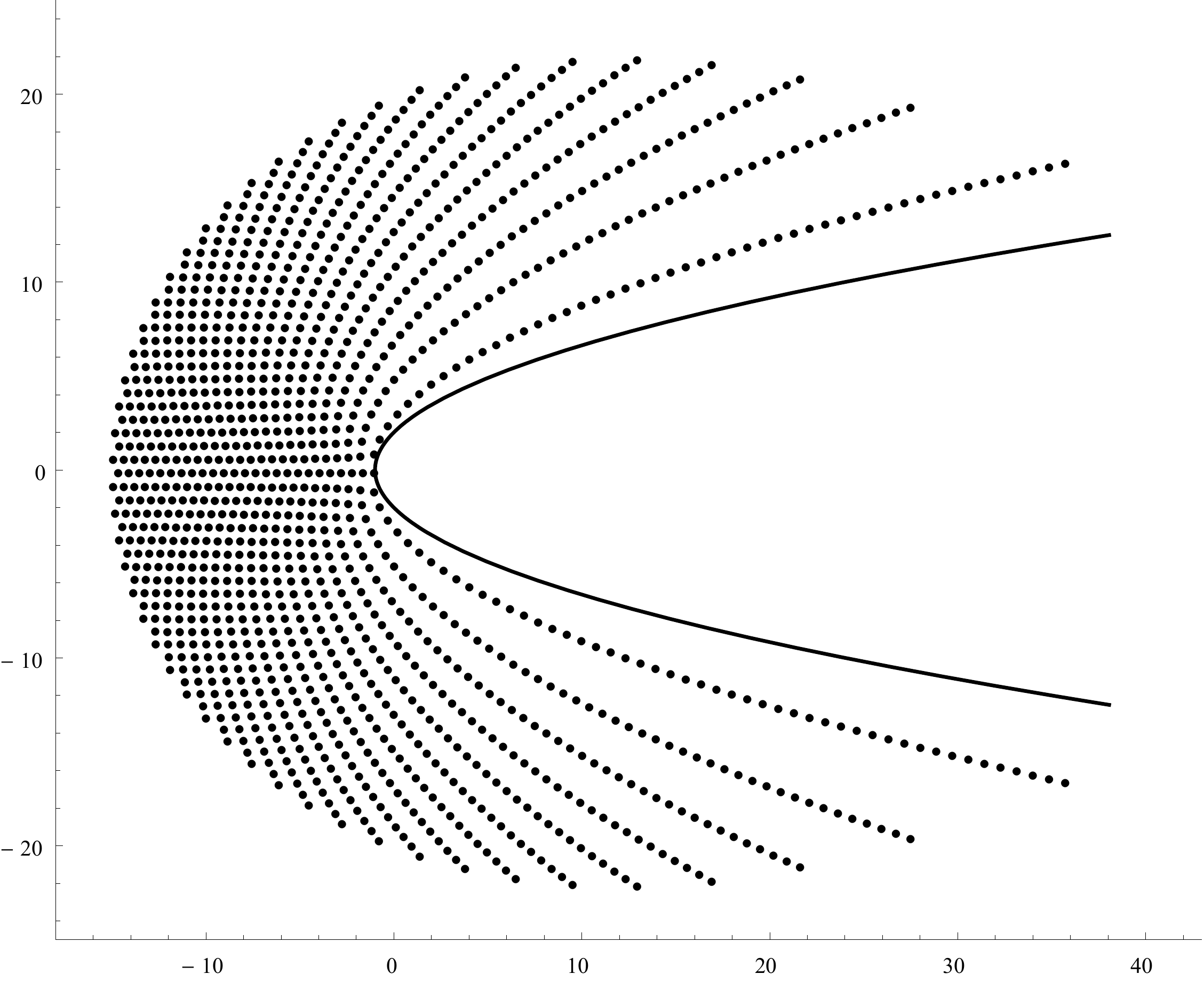}}
\end{figure}

\begin{theorema}[Newman-Rivlin]
\label{newrivtheo}
Let
\[
	p_n(z) = \sum_{k=0}^{n} \frac{z^k}{k!}.
\]
Then
\[
	\lim_{n \to \infty} \frac{p_n(n+w\sqrt{n})}{\exp(n+w\sqrt{n})} = \frac{1}{2} \erfc\!\left(w/\sqrt{2}\right)
\]
uniformly on any compact set in $\im w \geq 0$.
\end{theorema}

Here $\erfc$ refers to the complementary error function, which is defined in \eqref{erfc}.  As a consequence of this result it is possible to show that, for any positive constants $K$, $x_0$, and $\epsilon$, the set
\[
	\{x+iy : |y| \leq K x^{1/2+\epsilon} \text{ and } x \geq  x_0 \}
\]
contains infinitely many zeros of the partial sums of $\exp(z)$.  It is after this theorem that the present paper is named.

Prompted by these results and by additional numerical computations, Saff and Varga made the following conjecture (see \cite[p.\ 5]{esv:sections} and the references therein).

\begin{conjecturea}
\label{widthconjecture1}
Consider the ``parabolic region''
\[
	S_0(\tau) = \left\{z = x+iy : |y| \leq K x^{1-\tau/2}, \, x \geq x_0\right\},
\]
where $K$ and $x_0$ are fixed positive constants, and consider also the regions $S_\theta(\tau)$ obtained by rotations of $S_0(\tau)$:
\[
	S_\theta(\tau) = e^{i\theta} S_0(\tau).
\]
Given any entire function $f$ of positive finite order $\lambda > \tau$, denote its $n^\text{th}$ partial sum by $p_n(z)$.  There exists an infinite sequence of positive integers $N$ such that there is no $S_\theta(\tau)$ which is devoid of all zeros of all partial sums $p_n(z)$, $n \in N$.
\end{conjecturea}

Essentially this conjectures that any region free from zeros of the partial sums must not be too wide---this width depending on the exponential order of the function in question.  In the numerical evidence and the examples which have been investigated up to this point the zeros actually cluster densely together and ``fill'' up most of the plane, and there are only a finite number of exceptional arguments where these zero-free regions exist.  Further, these exceptional arguments only occur in the direction of maximal exponential growth of the function in question.  It is these arguments which we are interested in in the present paper.

To capture these observations, Edrei, Saff, and Varga proposed a modified Width Conjecture in \cite[p.\ 6]{esv:sections}.

\begin{conjectureb}
\label{widthconjecture2}
Let $f$ be an entire function of positive, finite order $\lambda$.  We can find an infinite sequence of positive integers $N$ and a finite number of exceptional arguments $\theta_1,\theta_2,\ldots,\theta_q$ such that
\begin{enumerate}[label=\textnormal{(\alph*)}]
\item For any argument $\theta \neq \theta_j$, $j = 1,2,\ldots,q$, it's possible to find a positive sequence $\rho_n$, $n \in N$, with $\rho_n \to \infty$ and $\rho_n = O(n^{2/\lambda})$ such that, for every fixed $\epsilon > 0$, the number of zeros of the partial sum $p_n(z)$ in the disk
\[
	\left|z - \rho_n e^{i\theta}\right| \leq \rho_n n^{-1+\epsilon}
\]
tends to infinity as $n \to \infty$, $n \in N$.

\item For any exceptional argument $\theta_j$ it's possible to find an integer $m \geq 2$ and a positive sequence $\rho_n$, $n \in N$, with $\rho_n \to \infty$ and $\rho_n = O(n^{2/(\lambda m)})$ such that, for every fixed $\epsilon > 0$, the number of zeros of the partial sum $p_n(z)$ in the disk
\[
	\left|z - \rho_n e^{i\theta_j}\right| \leq \rho_n n^{-1/m + \epsilon}
\]
tends to infinity as $n \to \infty$, $n \in N$.
\end{enumerate}
\end{conjectureb}

Results of the same type as Theorem \ref{newrivtheo} have so far been very important in verifying the Width Conjectures in these directions of maximal exponential growth.  Indeed, one can check that Theorem \ref{newrivtheo} verifies part (b) of the modified conjecture for the case of the exponential function with $\lambda = 1$, $m = 2$, $\theta = 0$, and $\rho_n = n$.  The following analogue of Theorem \ref{newrivtheo} is proved in \cite[p.\ 10]{esv:sections} also to verify part (b) of the modified conjecture at the exceptional argument $\theta = 0$ for the Mittag-Leffler functions.

\begin{theorema}[Edrei-Saff-Varga]
\label{esvtheo}
Let
\[
	E_{1/\lambda}(z) = \sum_{k=0}^\infty \frac{z^k}{\Gamma(k/\lambda + 1)}
\]
be the Mittag-Leffler function of positive, finite order $\lambda$.  Let $p_n(z)$ be its $n^\text{th}$ partial sum and let $r_n = (n/\lambda)^{1/\lambda} e^{1/(2n)}$.  Then
\[
	\lim_{n \to \infty} \frac{p_n\!\left(r_n \left(1+w\sqrt{2/(\lambda n)} \right)\right)}{\left(1+w\sqrt{2/(\lambda n)}\right)^n E_{1/\lambda}(r_n)} = \frac{1}{2} \exp\!\left(w^2\right) \erfc(w)
\]
uniformly for $w$ in any compact set in $\C$.
\end{theorema}

More results of this type can be found in \cite{norfolk:1f1, zhel:mlsectails, mallison:expsums, norfolk:binom, zhel:lindelof}. In the spirit of these and Theorems \ref{newrivtheo} and \ref{esvtheo} we prove the following general result.

\begin{theorem}
\label{maintheo}
Let $f$ be an entire function of positive, finite order $\lambda$ with asymptotic behavior as in \eqref{fgrowth} and which satisfies Condition \ref{dfgrowth}.  Let $r_n = (n/\lambda)^{1/\lambda}$ and let $p_n(z)$ denote the $n^\text{th}$ partial sum of the power series for $f(z)$.  Then
\[
	\lim_{n \to \infty} \frac{p_{n-1}(r_n(1+w/\sqrt{n}))}{f(r_n(1+w/\sqrt{n}))} = \frac{1}{2} \erfc\!\left(w\sqrt{\lambda/2}\,\right)
\]
uniformly on any compact set in $\re w < 0$.
\end{theorem}

We'll briefly describe how this result is connected with the Modified Width Conjecture.  If $w$ is any zero of $\tfrac{1}{2} \erfc(w\sqrt{\lambda/2})$ then $\re w < 0$, and by Hurwitz's theorem (see, e.g., \cite[p.\ 4]{marden:geom}) $p_{n-1}(z)$ has a zero $z_n$ of the form
\[
	z_n = r_n + w \frac{r_n}{\sqrt{n}} (1+o(1))
\]
for $n$ large enough.  As $n$ grows, this zero will eventually lie inside the disk
\[
	|z - r_n| \leq r_n n^{-1/2+\epsilon},
\]
where $\epsilon > 0$ is fixed.  The function $\erfc$ has infinitely many zeros, so the number of zeros in this disk will tend to infinity as $n \to \infty$.  This verifies part (b) of the Modified Width Conjecture with $\rho_n = r_n$, $\theta = 0$, and $m=2$ for this class of functions.

It is important to note that we will not claim to have shown that $\theta = 0$ is the only exceptional argument for the functions we consider or that part (a) of the conjecture has been resolved.  These questions are still open.  We also note that another variant of the Width Conjecture was proposed by Norfolk in \cite[p.\ 531]{norfolk:widthconj}, and it is straightforward to show that this too is satisfied (for a particular argument) by the above result.

To prove Theorem \ref{maintheo} we will adapt an approach involving Riemann-Hilbert methods introduced Kriecherbauer, Kuijlaars, McLaughlin, and Miller in \cite{mclaughlin:exprh} to study the zeros of the partial sums of $\exp(z)$.  In the paper the authors obtained strong asymptotics for each zero of the partial sums.  A crucial element in their approach is a Cauchy integral representation for these partial sums,
\[
	\frac{1}{2\pi i} \int_\gamma (se^{1-s})^{-n} \frac{ds}{s-z} = \begin{cases}
		-(ez)^{-n} \sum_{k=0}^{n-1} \frac{(nz)^k}{k!} & \text{for } z \text{ outside } \gamma, \\
		-(ez)^{-n} \sum_{k=0}^{n-1} \frac{(nz)^k}{k!} + (ze^{1-z})^{-n} & \text{for } z \neq 0 \text{ inside } \gamma.
		\end{cases}
\]
We make use of a more general version of this integral in \eqref{fdef} and \eqref{fnexplicit}.

\section{Definitions and Preliminaries}

Let $a,b \in \C$, $0 < \lambda < \infty$, $0 < \theta < \min\{\pi,\pi/\lambda\}$, and $\mu < 1$.  We suppose that $f$ is an entire function such that
\begin{equation}
\label{fgrowth}
f(z) = \begin{cases}
	z^a (\log z)^b \exp(z^\lambda)\bigl(1+o(1)\bigr) & \text{for } |\arg z| \leq \theta, \\
	O\!\left(\exp(\mu |z|^\lambda)\right) & \text{for } |\arg z| > \theta
	\end{cases}			
\end{equation}
as $|z| \to \infty$, with each estimate holding uniformly in its sector.  In this case $f$ is of exponential order $\lambda$.  For this $f$, let
\[
	p_n(z) = \sum_{k=0}^{n} \frac{f^{(k)}(0)}{k!} z^k.
\]

Define
\begin{equation}
\label{rndef}
	r_n = \left(\frac{n}{\lambda}\right)^{1/\lambda}.
\end{equation}
Note that for $|\arg z| \leq \theta$ we have
\begin{align}
\frac{f(r_n z)}{r_n^a (\log r_n)^b (e^{1/\lambda}z)^n} &\sim z^a \left(z^\lambda e^{1-z^\lambda}\right)^{-n/\lambda} \nonumber \\
&= z^a e^{n \varphi(z)} \label{integrandasymp}
\end{align}
as $n \to \infty$, where
\begin{equation}
\label{phidef}
	\varphi(z) = (z^\lambda - 1 - \log z^\lambda)/\lambda.
\end{equation}

Let $\Delta$ be the circle centered at $z=1$ which subtends an angle of $\theta$ from the origin.  Denote by $\sigma_1, \sigma_2$ the points where $\Delta$ intersects the line of steepest descent of the function $\re \varphi(z)$ passing through the point $z=1$.  Note that by symmetry $\sigma_1 = \overline{\sigma_2}$ and $\re \varphi(\sigma_1) = \re \varphi(\sigma_2)$.  Further, $\re \varphi(\sigma_1) < 0$.  We will impose the following growth condition on the derivative of the function $f$.

\begin{condition}
\label{dfgrowth}
There exists a constant $0 < \nu < -\re \varphi(\sigma_1)$ such that, if $z$ is restricted to any compact subset of $\{z \in \C : z \neq 0 \text{ and } |\arg z| \leq \theta\}$, we have
\[
	\frac{f'(r_n z)}{f(r_n z)} = O(e^{\nu n})
\]
uniformly in $z$ as $n \to \infty$.
\end{condition}

This technical condition is used in the proof of Lemma \ref{fnintlemma}.

\begin{definition}
\label{admisscontour}
A contour $\gamma$ is said to be \textit{admissible} if
\begin{enumerate}
\item $\gamma$ is a smooth Jordan curve winding counterclockwise around the origin.
\item In the sector $|\arg z| \leq \theta$, $\gamma$ is a positive distance from the curve $\re \varphi(z) = 0$ except for a part that lies in some neighborhood $U_\gamma$ of $z = 1$.  In this set $U_\gamma$ the contour $\gamma$ coincides with the path of steepest decent of the function $\re \varphi(z)$ passing through the point $z = 1$.
\item In the sector $|\arg z| \geq \theta$, $\gamma$ coincides with the unit circle.
\end{enumerate}
\end{definition}

We will now introduce a number of Cauchy-type integrals.  Various facts about this type of integral transform, including a detailed description of Sokhotski's formula, can be found in \cite[ch.\ 1]{gakhov:bvp}.

Let $\gamma$ be an admissible contour and suppose for now that $z \neq 0$ is inside $r_n \gamma$.  The function
\[
	\frac{f(z) - p_{n-1}(z)}{z^n} = \Phi(z)
\]
is entire, so by Cauchy's integral formula we have
\begin{equation}
\label{fnbuild}
	\Phi(z) = \frac{1}{2\pi i} \int_{r_n \gamma} \zeta^{-n} f(\zeta) \frac{d\zeta}{\zeta - z} - \frac{1}{2\pi i}\int_{r_n \gamma} \zeta^{-n} p_{n-1}(\zeta) \frac{d\zeta}{\zeta - z}.
\end{equation}
Since
\[
	\int_{r_n \gamma} \zeta^{-m} \frac{d\zeta}{\zeta - z} = 0
\]
for all integers $m \geq 1$, the second integral in \eqref{fnbuild} is zero.  Making the substitution $\zeta = r_n s$ yields the identity
\[
	\frac{f(r_n z) - p_{n-1}(r_n z)}{(r_n z)^n} = \frac{1}{2\pi i} \int_\gamma (r_n s)^{-n} f(r_n s) \frac{ds}{s-z},
\]
which holds for $z \neq 0$ inside $\gamma$.  (This construction is a special case of the one in \cite[p.\ 436]{edrei:paderem} for an integral representation of the error of a Pad\'e approximation.)

The above calculations motivate us to define the function
\begin{equation}
\label{fdef}
	F_n(z) = \frac{r_n^{-a} (\log r_n)^{-b}}{2\pi i} \int_{\gamma} (e^{1/\lambda}s)^{-n} f(r_n s) \frac{ds}{s-z}
\end{equation}
for all $z \notin \gamma$, $z \neq 0$.  For $z$ inside $\gamma$ with $z \neq 0$ we know from above that
\begin{equation}
\label{fnexplicit}
	F_n(z) = \frac{f(r_n z) - p_{n-1}(r_n z)}{r_n^a (\log r_n)^b (e^{1/\lambda} z)^n}.
\end{equation}
By Sokhotski's formula we have
\[
	F_n^+(z) = F_n^-(z) + \frac{f(r_n z)}{r_n^a (\log r_n)^b (e^{1/\lambda}z)^n}, \qquad z \in \gamma,
\]
where $F_n^+$ (resp.\ $F_n^-$) refers to the continuous extensions of $F_n$ from inside (resp.\ outside) $\gamma$ onto $\gamma$.  Though we don't need to for the present paper, we can also calculate $F_n(z)$ for $z$ outside $\gamma$ using the residue theorem.  In all,
\[
	F_n(z) = \frac{1}{r_n^a (\log r_n)^b (e^{1/\lambda} z)^n} \times \begin{cases}
		-p_{n-1}(r_n z) & \text{for } z \text{ outside } \gamma, \\
		f(r_n z) - p_{n-1}(r_n z) & \text{for } z \neq 0 \text{ inside } \gamma.
		\end{cases}
\]

Let $\gamma_\theta = \gamma \cap \{z \in \C : |\arg z| \leq \theta\}$ and define
\begin{equation}
\label{gdef}
	G_n(z) = \frac{1}{2\pi i}\int_{\gamma_\theta} e^{n \varphi(s)} \frac{ds}{s-z},
\end{equation}
where $\varphi$ is as in \eqref{phidef}.  Sokhotski's formula tells us that
\[
	G_n^+(z) = G_n^-(z) + e^{n\varphi(z)}, \qquad z \in \gamma_\theta,
\]
where $G_n^+$ and $G_n^-$ refer to the continuous extensions of $G_n$ from the left and right of $\gamma_\theta$ onto $\gamma_\theta$, respectively.  Based on the asymptotic \eqref{integrandasymp} and the fact that the saddle point of the function $\varphi(s)$ is located at $s=1$, we expect that $F_n(z) \approx G_n(z)$ for $z \approx 1$ as $n \to \infty$.  Something to this effect is shown in Lemma \ref{fngnapproxlemma}.

We observe that $\varphi(1) = \varphi'(1) = 0$ and $\varphi''(1) = \lambda$, so
\[
	\varphi(s) = \frac{\lambda}{2}(s-1)^2 + O\left((s-1)^3\right)
\]
in a neighborhood of $s=1$.  We can thus invoke the inverse function theorem to find a neighborhood $V$ of the origin, a neighborhood $U \subset U_\gamma$ of $s=1$, and a biholomorphic map $\psi : V \to U$ which satisfies
\[
	(\varphi \circ \psi)(\xi) = \xi^2
\]
for $\xi \in V$.  Note that the set $U_\gamma$ here is as defined in Definition \ref{admisscontour}.  This function $\psi$ maps a segment of the imaginary axis onto the path of steepest descent of the function $\re \varphi(z)$ going through $z=1$.

Just as in \cite[p.\ 189]{mclaughlin:exprh} we define
\[
	h(\zeta) = \frac{1}{2\pi i} \int_{-\infty}^{\infty} e^{-u^2} \frac{du}{u-\zeta}, \qquad \zeta \in \C \setminus \R
\]
and
\[
	P_n(z) = h\!\left(-i\sqrt{n} \psi^{-1}(z)\right), \qquad z \in U \setminus \gamma_\theta.
\]
By Sokhotski's formula we have
\[
	h^+(x) = h^-(x) + e^{-x^2}, \qquad x \in \R,
\]
and, setting $z = \psi(ix/\sqrt{n})$,
\[
	P_n^+(z) = P_n^-(z) + e^{n\varphi(z)}, \qquad z \in U \cap \gamma_\theta.
\]
Here $+$ and $-$ indicate approaching the contour $\gamma_\theta$ from the left and from the right, respectively.

Finally define
\begin{equation}
\label{erfc}
	\erfc(z) = \frac{2}{\sqrt{\pi}} \int_z^\infty e^{-s^2}\,ds,
\end{equation}
where the contour of integration is the horizontal line starting at $s=z$ and extending to the right to $s = z+\infty$.  This is known as the complementary error function.  For information about the zeros of this function we refer the reader to \cite{fettis:erfczeros}.

\section{Proof of the main result}

In this section we will prove Theorem \ref{maintheo}.

Choose $\epsilon > 0$ such that $\overline{B_{2\epsilon}(1)} \subset U$ and define
\[
	m(z) = \begin{cases}
		F_n(z) & \text{for } z \in \C \setminus \left( \gamma \cup \overline{B_{2\epsilon}(1)} \right), \\
		G_n(z) - P_n(z) & \text{for } z \in B_{2\epsilon}(1) \setminus \gamma.
	\end{cases}
\]
The jumps for $G_n(z)$ and $P_n(z)$ cancel each other out as $z$ moves across $\gamma$ in $B_{2\epsilon}(1)$, so $m$ is analytic on $B_{2\epsilon}(1)$.  If we define the contours
\[
	\Gamma_1 = \partial B_{2\epsilon}(1), \qquad \Gamma_2 = \gamma \setminus B_{2\epsilon}(1), \qquad \Gamma = \Gamma_1 \cup \Gamma_2,
\]
then the function $m$ uniquely solves the following Riemann-Hilbert problem.

\begin{figure}[!htb]
\floatbox[{\capbeside\thisfloatsetup{capbesideposition={right,center}, capbesidewidth=0.45\textwidth}}]{figure}[\FBwidth]
{\caption{Schematic representation of the new contour $\Gamma$.  The components $\Gamma_1$ and $\Gamma_2$ are indicated by a dashed line and a solid line, respectively.}\label{biggammaplot}}
{\includegraphics[width=0.5\textwidth]{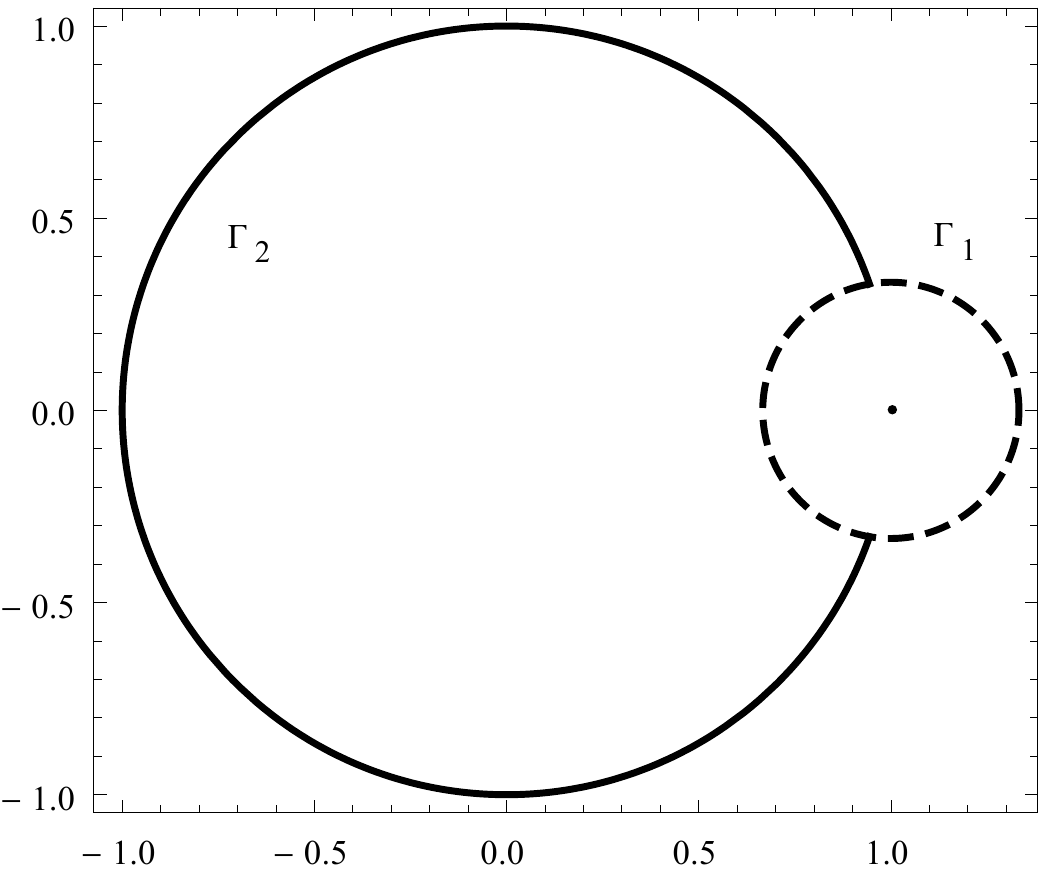}}
\end{figure}

\begin{rhp}
\label{mainrhp}
Seek an analytic function $M \colon \C \setminus \Gamma \to \C$ such that
\begin{enumerate}
 \item $M^+(z) = M^-(z) - P_n(z) + G_n(z) - F_n(z)$ for $z \in \Gamma_1 \setminus \Gamma_2$,
 \item $M^+(z) = M^-(z) + \frac{f(r_n z)}{r_n^a (\log r_n)^b (e^{1/\lambda}z)^n}$ for $z \in \Gamma_2$,
\item $M(z) \to 0$ as $|z| \to \infty$.
\end{enumerate}
\end{rhp}

We therefore have
\begin{align}
m(z) &= \frac{1}{2\pi i} \int_\Gamma \Bigl[m^+(s) - m^-(s)\Bigr] \frac{ds}{s-z} \nonumber \\
	&= -\frac{1}{2\pi i} \int_{\Gamma_1} P_n(s) \frac{ds}{s-z} + \frac{1}{2\pi i} \int_{\Gamma_1} G_n(s)\frac{ds}{s-z} - \frac{1}{2\pi i} \int_{\Gamma_1} F_n(s)\frac{ds}{s-z} \nonumber \\
	&\qquad + \frac{r_n^{-a} (\log r_n)^{-b}}{2\pi i} \int_{\Gamma_2} (e^{1/\lambda}s)^{-n} f(r_n s) \frac{ds}{s-z} \label{mint}
\end{align}
by Sokhotski's formula.  As $n \to \infty$, each of these integrals tends to zero uniformly as long as $z$ is bounded away from $\Gamma$.  Indeed, by referring to Lemmas \ref{gintlemma}, \ref{gamma2lemma}, \ref{fnintlemma}, and \ref{pintlemma}, we know that 
\[
	m(z) = o(1),
\]
and by the definition of $m$,
\[
	G_n(z) = P_n(z) + o(1)
\]
uniformly for $z \in B_\epsilon(1)$ as $n \to \infty$.  Now set $z = 1 + w/\sqrt{n}$, where $w$ is restricted to a compact subset of $\re w < 0$.  By Lemma \ref{fngnapproxlemma} we deduce from the above that
\begin{equation}
\label{fnpnasymp}
	F_n(1+w/\sqrt{n}) = P_n(1+w/\sqrt{n}) + o(1)
\end{equation}
uniformly as $n \to \infty$.

Following the argument in \cite[p.\ 194]{mclaughlin:exprh}, it's possible to show that
\[
	h(\zeta) = \frac{1}{2}e^{-\zeta^2} \erfc(-i\zeta)
\]
on $\im \zeta > 0$.  Setting
\[
	\zeta = -i\sqrt{n}\psi^{-1}(z) = -i\sqrt{n\varphi(z)}
\]
for an appropriately chosen branch of the square root we obtain an expression for $P_n$,
\[
	P_n(z) = \frac{1}{2} e^{n\varphi(z)} \erfc\!\left(-\sqrt{n\varphi(z)}\right),
\]
valid for $z \in U$ to the left of $\gamma_\theta$.  Since $2 - \erfc(x) = \erfc(-x)$ we can rewrite this as
\[
	P_n(z) = e^{n\varphi(z)} - \frac{1}{2} e^{n\varphi(z)} \erfc\!\left(\sqrt{n\varphi(z)}\right).
\]
It is straightforward to show that
\[
	\lim_{n \to \infty} n \varphi(1+w/\sqrt{n}) = \frac{\lambda}{2} w^2
\]
uniformly, so
\[
	P_n(1+w/\sqrt{n}) = e^{\lambda w^2/2} - \frac{1}{2} e^{\lambda w^2/2} \erfc\!\left(w \sqrt{\lambda/2}\,\right) + o(1)
\]
uniformly as $n \to \infty$.  By substituting this into \eqref{fnpnasymp} we see that
\begin{equation}
\label{fnerfasymp1}
F_n(1+w/\sqrt{n}) = e^{\lambda w^2/2} - \frac{1}{2} e^{\lambda w^2/2} \erfc\!\left(w \sqrt{\lambda/2}\,\right) + o(1)
\end{equation}
uniformly as $n \to \infty$.

For $n$ large enough we can write
\[
	F_n(1+w/\sqrt{n}) = \frac{1}{r_n^a (\log r_n)^b} \left(\frac{f(r_n(1+w/\sqrt{n}))}{e^{n/\lambda} (1+w/\sqrt{n})^{n}} - \frac{p_{n-1}(r_n(1+w/\sqrt{n}))}{e^{n/\lambda} (1+w/\sqrt{n})^{n}} \right)
\]
by \eqref{fnexplicit}.  The asymptotic assumption \eqref{fgrowth} grants us the uniform estimate
\[
	\frac{f(r_n(1+w/\sqrt{n}))}{r_n^a (\log r_n)^b e^{n/\lambda} (1+w/\sqrt{n})^{n}} = e^{\lambda w^2/2} + o(1),
\]
and upon substituting this into the above formula we find that
\[
	F_n(1+w/\sqrt{n}) = e^{\lambda w^2/2} - e^{\lambda w^2/2} \frac{p_{n-1}(r_n(1+w/\sqrt{n}))}{f(r_n(1+w/\sqrt{n}))}(1+o(1)) + o(1)
\]
uniformly as $n \to \infty$.  Substituting this into \eqref{fnerfasymp1} yields the expression
\[
\frac{p_{n-1}(r_n(1+w/\sqrt{n}))}{f(r_n(1+w/\sqrt{n}))} (1+o(1)) = \frac{1}{2} \erfc\!\left(w \sqrt{\lambda/2}\,\right) + o(1),
\]
which holds uniformly as $n \to \infty$.  Theorem \ref{maintheo} follows immediately from this asymptotic.

\section{Lemmas from the proof}

\begin{lemma}
\label{gintlemma}
\[
	\int_{\Gamma_1} G_n(s)\frac{ds}{s-z} = O(n^{-1/2})
\]
uniformly for $z \in B_\epsilon(1)$ as $n \to \infty$.
\end{lemma}

\begin{proof}
For $z \in B_\epsilon(1)$ we have
\begin{align*}
	\left| \int_{\Gamma_1} G_n(s)\frac{ds}{s-z} \right| &\leq 4\pi \epsilon \left\| \frac{G_n(s)}{s-z} \right\|_{L^\infty(\Gamma_1)} \\
	&\leq 4\pi \| G_n(s) \|_{L^\infty(\Gamma_1)}.
\end{align*}
Let $\Gamma_1^+$ and $\Gamma_1^-$ denote the closures of the parts of $\Gamma_1$ lying to the left and to the right of $\gamma_\theta$, respectively.  Then from the above we see that
\begin{equation}
\label{gintest}
	\left| \int_{\Gamma_1} G_n(s)\frac{ds}{s-z} \right| \leq 4\pi \left( \| G_n(s) \|_{L^\infty(\Gamma_1^+)} + \| G_n(s) \|_{L^\infty(\Gamma_1^-)} \right).
\end{equation}
Define $s_1,s_2$ to be the points where $\Gamma_1$ intersects $\gamma_\theta$.

Depending on whether $s$ approaches $s_j$ from the left or the right, we have
\[
	G_n(s_j) = \pm \frac{1}{2} e^{n\varphi(s_j)} + \frac{1}{2\pi i} \operatorname{P.V.} \int_{\gamma_\theta} e^{n\varphi(t)} \frac{dt}{t-s_j}.
\]
Note that the first term here decays exponentially.  We can deform the contour $\gamma_\theta$ in a small neighborhood $A$ of $s_j$ to be a straight line passing through $s_j$.  Choose this neighborhood small enough so that $\gamma_\theta$ still lies entirely below the saddle point at $s=1$ on the surface $\re \varphi(s)$ except where it passes through $s=1$.  We then have
\[
	\operatorname{P.V.} \int_{\gamma_\theta} e^{n\varphi(t)} \frac{dt}{t-s_j} = \int_{\gamma_\theta \cap A} \frac{e^{n\varphi(t)} - e^{n\varphi(s_j)}}{t-s_j}\,dt + \int_{\gamma_\theta \setminus A} e^{n\varphi(t)} \frac{dt}{t-s_j}.
\]
A straightforward application of the Laplace method to the second integral here yields
\[
	\int_{\gamma_\theta \setminus A} e^{n\varphi(t)} \frac{dt}{t-s_j} = O(n^{-1/2}).
\]
From Taylor's theorem we know that
\begin{align*}
\left| e^{n\varphi(t)} - e^{n\varphi(s_j)} \right| &\leq |t-s_j| \sup_{\tau \in \gamma_\theta \cap A} \left| n\varphi'(\tau) e^{n\varphi(\tau)} \right| \\
&\leq |t-s_j| n e^{n(\re \varphi(s_j)+c)} \sup_{\tau \in \gamma_\theta} |\varphi'(\tau)|,
\end{align*}
where $0 < c < -\re \varphi(s_j)$.  From this it follows that
\[
	\left| \int_{\gamma_\theta \cap A} \frac{e^{n\varphi(t)} - e^{n\varphi(s_j)}}{t-s_j}\,dt \right| < \const. \cdot ne^{n(\re \varphi(s_j) + c)},
\]
and this tends to $0$.  Combining these facts we conclude that
\begin{equation}
\label{gest1}
	G_n(s_j) = O(n^{-1/2})
\end{equation}
as $n \to \infty$.

Now suppose $s \in \Gamma_1^+ \setminus \{s_1,s_2\}$.  Then $e^{n\varphi(t)}/(t-s)$ is analytic in a neighborhood of $\gamma_\theta$.  We can deform $\gamma_\theta$ near $s_1$ and $s_2$ so that it stays a small positive distance away from $\Gamma_1^+$, and in such a way that $\gamma_\theta$ is unchanged in the disk $B_\epsilon(1)$.  Split the integral for $G_n(s)$ into the pieces
\[
	G_n(s) = \frac{1}{2\pi i}\int_{\gamma_\theta \setminus B_\epsilon(1)} e^{n\varphi(t)} \frac{ds}{t-s} + \frac{1}{2\pi i}\int_{\gamma_\theta \cap B_\epsilon(1)} e^{n\varphi(t)} \frac{ds}{t-s}.
\]
After this deformation, the first integral is bounded by
\[
	\left| \int_{\gamma_\theta \setminus B_\epsilon(1)} e^{n\varphi(t)} \frac{ds}{t-s} \right| \leq C e^{-c n},
\]
where $C > 0$ and $c > 0$ are constants independent of $s$.  In the second integral let $t = \psi(iu)$ and define $-i\psi^{-1}(\gamma_\theta \cap B_\epsilon(1)) = (-\alpha_1,\alpha_2)$, so that
\begin{align}
\left| \int_{\gamma_\theta \cap B_\epsilon(1)} e^{n\varphi(t)} \frac{ds}{t-s} \right| &= \left| \int_{-\alpha_1}^{\alpha_2} e^{-nu^2} \frac{i\psi'(iu)}{\psi(iu)-s}\,du \right| \nonumber \\
&\leq \sup_{u \in (\alpha_1,\alpha_2)} \left| \frac{\psi'(iu)}{\psi(iu)-s} \right| \int_{-\alpha_1}^{\alpha_2} e^{-nu^2}\,du \nonumber \\
&\leq \epsilon^{-1} \sqrt{\pi/n} \sup_{u \in (\alpha_1,\alpha_2)} |\psi'(iu)|. \label{gest2}
\end{align}
An identical process will yield the same bound for $s \in \Gamma_1^- \setminus \{s_1,s_2\}$.

Combining \eqref{gest1} and \eqref{gest2} in \eqref{gintest}, we conclude that
\begin{equation}
\label{gestfinish}
	\int_{\Gamma_1} G_n(s) \frac{ds}{s-z} = O(n^{-1/2})
\end{equation}
uniformly for $z \in B_\epsilon(1)$ as $n \to \infty$.
\end{proof}

\begin{lemma}
\label{gamma2lemma}
There exists a constant $c>0$ such that
\[
	r_n^{-a} (\log r_n)^{-b} \int_{\Gamma_2} (e^{1/\lambda}s)^{-n} f(r_n s) \frac{ds}{s-z} = O(e^{-cn})
\]
uniformly for $z \in B_\epsilon(1)$ as $n \to \infty$.
\end{lemma}

\begin{proof}
Let $\Gamma_2'$ denote the part of $\Gamma_2$ for which $|\arg s| \leq \theta$ and let $\Gamma_2''$ denote the part for which $\theta < |\arg s|$.  We'll split the integral into the two parts
\[
	\int_{\Gamma_2} = \int_{\Gamma_2'} + \int_{\Gamma_2''}
\]
and estimate them separately.

For $|\arg z| \leq \theta$ we can write
\[
	f(z) = z^a (\log z)^b \exp(z^\lambda)(1 + \delta(z)),
\]
where $\delta(z) \to 0$ uniformly as $|z| \to \infty$, so for $s \in \Gamma_2'$ we have
\[
	\frac{f(r_n s)}{r_n^a (\log r_n)^b (e^{1/\lambda}s)^n} = s^a e^{n \varphi(s)} \left(1 + \frac{\log s}{\log r_n}\right)^b (1 + \delta(r_n s)).
\]
If $s \in \Gamma_2'$ then we can find a constant $d > 0$ such that $\re \varphi(s) < -d$.  The quantities $s^a$, $\log s/\log r_n$, and $\delta(r_n s)$ are uniformly bounded for $s \in \Gamma_2'$, and the quantities $r_n^a$ and $(\log r_n)^b$ grow subexponentially, so if $z \in B_\epsilon(1)$ we can find positive constants $C_1$ and $d'$ such that
\[
	\left| \int_{\Gamma_2'} (e^{1/\lambda}s)^{-n} f(r_n s) \frac{ds}{s-z} \right| \leq \len(\Gamma_2') \cdot \epsilon^{-1} \cdot C_1 e^{-d'n}.
\]

For $|\arg z| > \theta$ we can write
\[
	|f(z)| \leq C_2 \exp(\mu|z|^\lambda)
\]
for some constant $C_2$.  If $s \in \Gamma_2''$ then $|s| = 1$, so
\[
	\left| \frac{f(r_n s)}{(e^{1/\lambda}s)^n} \right| \leq C_2 \exp[(\mu - 1)n/\lambda],
\]
and, since $|s-z| \geq \epsilon$,
\[
	\left| \int_{\Gamma_2''} (e^{1/\lambda}s)^{-n} f(r_n s) \frac{ds}{s-z} \right| \leq \len(\Gamma_2'') \cdot \epsilon^{-1} \cdot C_2 \exp[(\mu - 1)n/\lambda].
\]
Combining this with the above estimate yields the desired result.
\end{proof}

\begin{lemma}
\label{fnintlemma}
\[
	\int_{\Gamma_1} F_n(s)\frac{ds}{s-z} = O(n^{-1/2})
\]
uniformly for $z \in B_\epsilon(1)$ as $n \to \infty$.
\end{lemma}

\begin{proof}
Split the integral for $F_n$ into the two pieces
\[
	F_n(s) = \frac{r_n^{-a} (\log r_n)^{-b}}{2\pi i} \left( \int_{\gamma \setminus \gamma_\theta} (e^{1/\lambda}t)^{-n} f(r_nt) \frac{dt}{t-s} + \int_{\gamma_\theta} (e^{1/\lambda}t)^{-n} f(r_nt) \frac{dt}{t-s} \right)
\]
and denote by $F_n^1(s)$ and $F_n^2(s)$ the left and right terms, respectively.

If $s \in \Gamma_1$ and $t \in \gamma\setminus\gamma_\theta$ then $|t-s| \geq C_1$ for some constant $C_1 > 0$ since
\[
	s \in \overline{B_{2\epsilon}(1)} \subset U \subset \{z \in \C : |\arg z| \leq \theta\}
\]
and $U$ is open.  We can find a constant $C_2$ such that
\[
	|f(z)| \leq C_2 \exp(\mu|z|^\lambda)
\]
for $|\arg z| \geq \theta$, and just as in the proof of Lemma \ref{gamma2lemma} we get
\[
	\left| \int_{\gamma \setminus \gamma_\theta} (e^{1/\lambda}t)^{-n} f(r_nt) \frac{dt}{t-s} \right| \leq \len(\gamma \setminus \gamma_\theta) \cdot C_1^{-1} \cdot C_2 \exp[(\mu-1)n/\lambda].
\]
It follows that there are positive constants $C_3$ and $c$ such that
\[
	\left| \int_{\Gamma_1} F_n^1(s) \frac{ds}{s-z} \right| \leq C_3 e^{-cn}.
\]

Now we consider the integral over $\gamma_\theta$.  For $|\arg z| \leq \theta$ we can write
\begin{equation}
\label{deltadef}
	f(z) = z^a (\log z)^b \exp(z^\lambda)(1+\delta(z)),
\end{equation}
where $\delta(z) \to 0$ uniformly as $|z| \to \infty$.  This implies
\[
	\frac{f(r_n t)}{r_n^a (\log r_n)^b (e^{1/\lambda}t)^n} = t^a e^{n \varphi(t)} \left(1 + \frac{\log t}{\log r_n}\right)^b (1 + \delta(r_n t))
\]
for $t \in \gamma_\theta$, so we will rewrite
\begin{align*}
	&\int_{\Gamma_1} F^2_n(s) \frac{ds}{s-z} \\
	&\qquad = \frac{1}{2\pi i}\int_{\Gamma_1} \frac{1}{s-z} \int_{\gamma_\theta} t^a e^{n\varphi(t)} \frac{dt}{t-s} \,ds + \frac{1}{2\pi i}\int_{\Gamma_1} \frac{1}{s-z} \int_{\gamma_\theta} t^a e^{n\varphi(t)} \tilde{\delta}(r_n, t) \frac{dt}{t-s} \,ds,
\end{align*}
where
\begin{equation}
\label{deltatdef}
	\tilde{\delta}(r_n, t) = \left(1 + \frac{\log t}{\log r_n}\right)^b (1 + \delta(r_n t)) - 1.
\end{equation}
The first integral in this expression can be estimated using the method in Lemma \ref{gintlemma} while the second requires a little more care.  Actually the proof will go through just as before except for the estimates at the points $s_j$, which we will detail here.

Let's name the inner integral
\[
	g_n(s) = \frac{1}{2\pi i}\int_{\gamma_\theta} t^a e^{n\varphi(t)} \tilde{\delta}(r_n, t) \frac{dt}{t-s}.
\]
Depending on whether $s$ approaches $s_j$ from the left or the right, we have
\[
	g_n(s_j) = \pm \frac{1}{2} s_j^a e^{n\varphi(s_j)}\tilde{\delta}(r_n, s_j) + \frac{1}{2\pi i} \operatorname{P.V.} \int_{\gamma_\theta} t^a e^{n\varphi(t)} \tilde{\delta}(r_n, t) \frac{dt}{t-s_j}.
\]
The first term here decays exponentially.  We can deform the contour $\gamma_\theta$ in a small neighborhood $A$ of $s_j$ to be a straight line passing through $s_j$.  Choose this neighborhood small enough so that $\gamma_\theta$ still lies entirely below the saddle point at $s=1$ on the surface $\re \varphi(s)$ except where it passes through $s=1$.  We then have
\begin{align*}
	&\operatorname{P.V.} \int_{\gamma_\theta} t^a e^{n\varphi(t)} \tilde{\delta}(r_n, t) \frac{dt}{t-s_j} \\
	&\qquad = \int_{\gamma_\theta \cap A} \frac{t^a e^{n\varphi(t)} \tilde{\delta}(r_n, t) - s_j^a e^{n\varphi(s_j)}\tilde{\delta}(r_n, s_j)}{t-s_j}\,dt + \int_{\gamma_\theta \setminus A} t^a e^{n\varphi(t)} \tilde{\delta}(r_n, t) \frac{dt}{t-s_j}.
\end{align*}
For the second integral, the Laplace method yields
\[
	\int_{\gamma_\theta \setminus A} t^a e^{n\varphi(t)} \tilde{\delta}(r_n, t) \frac{dt}{t-s_j} = o(n^{-1/2}).
\]
From Taylor's theorem we know that
\begin{align*}
	&\left| t^a e^{n\varphi(t)} \tilde{\delta}(r_n, t) - s_j^a e^{n\varphi(s_j)}\tilde{\delta}(r_n, s_j) \right| \\
	&\qquad \leq |t-s_j| \sup_{\tau \in \gamma_\theta \cap A} \left|a\tau^{a-1}e^{n\varphi(\tau)}\tilde{\delta}(r_n, \tau) + n\varphi'(\tau)\tau^ae^{n\varphi(\tau)}\tilde{\delta}(r_n, \tau) + \tau^ae^{n\varphi(\tau)}\tilde{\delta}_\tau(r_n, \tau) \right| \\
	&\qquad \leq |t-s_j| \left(\sup_{\tau \in \gamma_\theta \cap A} \left|a\tau^{a-1}e^{n\varphi(\tau)}\tilde{\delta}(r_n, \tau) + n\varphi'(\tau)\tau^ae^{n\varphi(\tau)}\tilde{\delta}(r_n, \tau) \right| \right. \\
	&\hspace{3cm} + \left. \sup_{\tau \in \gamma_\theta \cap A} \left| \tau^ae^{n\varphi(\tau)}\tilde{\delta}_\tau(r_n, \tau) \right| \right).
\end{align*}
The first supremum here decays exponentially.  For the second we have
\[
	\sup_{\tau \in \gamma_\theta \cap A} \left| \tau^a e^{n\varphi(\tau)}\tilde{\delta}_\tau(r_n, \tau) \right| \leq e^{n(\re \varphi(s_j)+c')} \sup_{\tau \in \gamma_\theta} \left| \tau^a \tilde{\delta}_\tau(r_n, \tau) \right|,
\]
where $0 < c' < -\re \varphi(s_j)$.  By choosing $A$ smaller we can show that this estimate holds for any fixed $c' > 0$ small enough.  We calculate
\[
	\tilde{\delta}_\tau(r_n, \tau) = \left(\frac{b}{\tau\log(r_n \tau)} + \frac{r_n}{1+\delta(r_n \tau)} \right) \left( \tilde\delta(r_n,\tau) + 1 \right) \delta'(r_n \tau)
\]
and, from \eqref{deltadef},
\[
	\delta'(r_n \tau) = \left[ \frac{f'(r_n \tau)}{f(r_n \tau)} - \frac{1}{r_n \tau} \left( a + \frac{b}{\log(r_n \tau)} + n\tau^\lambda \right) \right](1+\delta(r_n \tau)).
\]
After substituting this into the previous expression, we may now appeal to Condition \ref{dfgrowth} to write
\[
	\sup_{\tau \in \gamma_\theta \cap A} \left| \tau^ae^{n\varphi(\tau)}\tilde\delta_\tau(r_n,\tau) \right| \leq C_4 r_n e^{n (\re \varphi(s_j) + c' + \nu)},
\]
where $C_4 > 0$ is a constant independent of $n$.  In addition to taking $c'$ as small as we like, by choosing $U_\gamma$, $U$, and $\epsilon$ slightly larger we may make $\re \varphi(s_j)$ as close to $\re \varphi(\sigma_j)$ as we like. We can thus make arrangements so that the quantity $\re \varphi(s_j) + c' + \nu$ is negative.  It follows that
\[
	\left| \int_{\gamma_\theta \cap A} \frac{t^a e^{n\varphi(t)} \tilde{\delta}(r_n, t) - s_j^a e^{n\varphi(s_j)}\tilde{\delta}(r_n, s_j)}{t-s_j}\,dt \right| \leq C_5 e^{-c''n}
\]
for some positive constants $C_5$ and $c''$, and combining this with the above Laplace method estimate we find that
\[
	g_n(s_j) = o(n^{-1/2})
\]
as $n \to \infty$.

The remainder of the proof proceeds exactly as in Lemma \ref{gintlemma}.
\end{proof}

\begin{lemma}
\label{pintlemma}
\[
	\int_{\Gamma_1} P_n(s) \frac{ds}{s-z} = O(n^{-1/2})
\]
uniformly for $z \in B_\epsilon(1)$ as $n \to \infty$.
\end{lemma}

\begin{proof}
We can find a constant $C_1$ such that
\[
	|h(\zeta)| \leq C_1 |\zeta|^{-1}
\]
for $\zeta \notin \R$.  Setting $\zeta = -i\sqrt{n}\psi^{-1}(s)$ yields
\[
	|P_n(s)| \leq C_1 n^{-1/2} |\psi^{-1}(s)|^{-1} = C_1 n^{-1/2} |\varphi(s)|^{-1/2}
\]
for $s \in U \setminus \gamma_\theta$.  Thus if $s \in \Gamma_1$ then $|\varphi(s)| \geq C_2$ for some constant $C_2 > 0$, so
\[
	\left| \int_{\Gamma_1} P_n(s) \frac{ds}{s-z} \right| \leq C_1 C_2^{-1/2} \epsilon^{-1} n^{-1/2}
\]
for $z \in B_\epsilon(1)$.
\end{proof}

\begin{lemma}
\label{fngnapproxlemma}
\[
	\lim_{n \to \infty} F_n(1+w/\sqrt{n}) - G_n(1+w/\sqrt{n}) = 0
\]
uniformly for $w$ restricted to compact subsets of $\re w < 0$.
\end{lemma}

\begin{proof}
In this proof we will write $z = 1+w/\sqrt{n}$ as a shorthand, keeping in mind the implicit dependence of $z$ on $n$.

Split the integral for $F_n$ into the two pieces
\[
	F_n(z) = \frac{r_n^{-a} (\log r_n)^{-b}}{2\pi i} \left( \int_{\gamma_\theta} (e^{1/\lambda}s)^{-n} f(r_ns) \frac{ds}{s-z} + \int_{\gamma \setminus \gamma_\theta} (e^{1/\lambda}s)^{-n} f(r_ns) \frac{ds}{s-z} \right).
\]
As in the previous lemmas, the second integral here is uniformly exponentially decreasing, and we can write the integrand of the first as
\[
	e^{n\varphi(s)} + e^{n\varphi(s)} \left(s^a - 1\right) + s^a e^{n\varphi(s)} \tilde{\delta}(r_n, s),
\]
where $\tilde\delta$ is as defined in \eqref{deltatdef}, to get
\begin{align}
F_n(s) = G_n(z) + \frac{1}{2\pi i} \int_{\gamma_\theta} e^{n\varphi(s)} \left(s^a - 1\right) \frac{ds}{s-z} + \frac{1}{2\pi i} \int_{\gamma_\theta} s^a e^{n\varphi(s)}\tilde{\delta}(r_n, s) \frac{ds}{s-z} + O(e^{-cn})
\label{fngnequiv}
\end{align}
for some constant $c > 0$.  We will show that both of these remaining integrals tend to $0$ uniformly.

The contour $\gamma_\theta$ passes through the point $s=1$ vertically, so by assumption there exists a positive constant $C_2$ such that $|s-z| \geq C_1 n^{-1/2}$.  For $n$ large enough $z \notin \gamma_\theta$, and in that case we have
\[
	\left|\frac{s-1}{s-z}\right| \leq 1 + \left| \frac{1-z}{s-z} \right| \leq 1 + C_1^{-1} n^{1/2} |1-z| \leq C_2
\]
for some constant $C_2$.  We then have
\begin{align*}
\left| \int_{\gamma_\theta} e^{n\varphi(s)} \left(s^a - 1\right) \frac{ds}{s-z} \right| &\leq \int_{\gamma_\theta} e^{n\re \varphi(s)} \left| \frac{s^a - 1}{s-1} \right| \left| \frac{s-1}{s-z} \right| |ds| \\
	&\leq C_2 \int_{\gamma_\theta} e^{n\re \varphi(s)} \left| \frac{s^a - 1}{s-1} \right| |ds|,
\end{align*}
which tends to zero as $n \to \infty$.

Split the second integral in \eqref{fngnequiv} like
\[
	\int_{\gamma_\theta} = \int_{\gamma_\theta \cap B_\epsilon(1)} + \int_{\gamma_\theta \setminus B_\epsilon(1)}.
\]
The integral over $\gamma_\theta \setminus B_\epsilon(1)$ decreases exponentially.  Let $s = \psi(it)$ and let
\[
	-i\psi^{-1}(\gamma_\theta \cap B_\epsilon(1)) = (-\alpha_1,\alpha_2),
\]
so that
\begin{align*}
&\left| \int_{\gamma_\theta \cap B_\epsilon(1)} s^a e^{n\varphi(s)}\tilde{\delta}(r_n, s) \frac{ds}{s-z} \right| \\
&\qquad = \left| \int_{-\alpha_1}^{\alpha_2} e^{-nt^2} \tilde{\delta}(r_n, \psi(it)) \frac{\psi(it)^a \psi'(it)}{\psi(it) - z}\,dt \right| \\
&\qquad \leq C_1^{-1} n^{1/2} \sup_{-\alpha_1 < t < \alpha_1} \left|\tilde{\delta}(r_n, \psi(it)) \psi(it)^a\psi'(it)\right| \int_{-\alpha_1}^{\alpha_2} e^{-nt^2} \,dt \\
&\qquad < C_1^{-1} \sqrt{\pi} \sup_{-\alpha_1 < t < \alpha_1} \left|\tilde{\delta}(r_n, \psi(it)) \psi(it)^a\psi'(it)\right|,
\end{align*}
which tends to $0$ as $n \to \infty$ by our assumption on $\delta$ and, by extension, $\tilde\delta$.

Combining the above estimates with \eqref{fngnequiv} we find that
\[
	F_n(z) = G_n(z) + o(1)
\]
uniformly as $n \to \infty$.
\end{proof}

\section{Discussion of the asymptotic assumption on $f$}

The assumption in \eqref{fgrowth} that our function $f$ has only one direction of maximal exponential growth is made in part to simplify the discussion.  It should not be an issue to extend the result to entire functions which have maximal growth along a set of arguments $\theta_1,\ldots,\theta_m$ with $\theta_j \neq \theta_k \pmod \pi$ for $j \neq k$.  However, we know from our results in \cite{vargas:limitcurves} that there are entire functions which grow maximally in two opposite directions whose partial sums cannot have the asymptotic behavior described in Theorem \ref{maintheo}.

The function
\[
	f(z) = \int_{-1}^{1} (1-t) e^{zt}\,dt = \frac{e^z - e^{-z}(1+2z)}{z^2}
\]
is one such example.  This function has maximal exponential growth along the arguments $\theta = 0,\pi$.  From \cite[pp.\ 225-226]{vargas:limitcurves} we know that in the right half-plane the zeros of its scaled partial sums $p_n(nz)$ approach the Szeg\H{o} curve $|ze^{1-z}| = 1$  from the inside, and so, since the Szeg\H{o} curve comes to a right angle at the point $z=1$, asymptotically satisfy the inequality $|\arg(z-1)| > 3\pi/4$.  However, all zeros of the complementary error function $\erfc(z)$ lie in the sector $|\arg z| < 3\pi/4$, hence the zeros of the partial sums cannot be related to zeros of the complementary error function in the way guaranteed by Theorem \ref{maintheo}.  It is unclear whether the method can be modified to handle cases such as these.

\begin{acknow}
I would like to thank my supervisor, Karl Dilcher, for introducing me to this problem and Robert Milson for his encouragement and for many valuable discussions.
\end{acknow}

\bibliographystyle{amsplain}
\bibliography{biblio}

\providecommand{\bysame}{\leavevmode\hbox to3em{\hrulefill}\thinspace}
\providecommand{\MR}{\relax\ifhmode\unskip\space\fi MR }
% \MRhref is called by the amsart/book/proc definition of \MR.
\providecommand{\MRhref}[2]{%
  \href{http://www.ams.org/mathscinet-getitem?mr=#1}{#2}
}
\providecommand{\href}[2]{#2}
\begin{thebibliography}{10}

\bibitem{mallison:expsums}
P.~Bleher and R.~{Mallison, Jr.}, \emph{Zeros of sections of exponential sums},
  Int. Math. Res. Not. (2006), Art. ID 38937, 49.

\bibitem{edrei:paderem}
A.~Edrei, \emph{The {P}ad\'e table of functions having a finite number of
  essential singularities}, Pacific J. Math. \textbf{56} (1975), no.~2,
  429--453.

\bibitem{esv:sections}
A.~Edrei, E.~B. Saff, and R.~S. Varga, \emph{Zeros of {S}ections of {P}ower
  {S}eries}, Lecture {N}otes in {M}athematics, vol. 1002, Springer-Verlag,
  Berlin, 1983.

\bibitem{fettis:erfczeros}
H.~E. Fettis, J.~C. Caslin, and K.~R. Cramer, \emph{Complex zeros of the error
  function and of the complementary error function}, Math. Comp. \textbf{27}
  (1973), 401--407.

\bibitem{gakhov:bvp}
F.~D. Gakhov, \emph{Boundary {V}alue {P}roblems}, Translation edited by I. N.
  Sneddon, Pergamon Press, Oxford-New York-Paris; Addison-Wesley Publishing
  Co., Inc., Reading, Mass.-London, 1966.

\bibitem{iverson:zeros}
K.~E. Iverson, \emph{The zeros of the partial sums of {$e^z$}}, Math. Tables
  and Other Aids to Computation \textbf{7} (1953), 165--168.

\bibitem{norfolk:binom}
S.~Janson and T.~S. Norfolk, \emph{Zeros of sections of the binomial
  expansion}, Electron. Trans. Numer. Anal. \textbf{36} (2009/10), 27--38.

\bibitem{mclaughlin:exprh}
T.~Kriecherbauer, A.~B.~J. Kuijlaars, K.~D. T.-R. McLaughlin, and P.~D. Miller,
  \emph{Locating the zeros of partial sums of {$e^z$} with {R}iemann-{H}ilbert
  methods}, Integrable {S}ystems and {R}andom {M}atrices, Contemp. Math., vol.
  458, Amer. Math. Soc., Providence, RI, 2008, pp.~183--195.

\bibitem{marden:geom}
M.~Marden, \emph{Geometry of {P}olynomials}, Second edition. Mathematical
  Surveys, No. 3, American Mathematical Society, Providence, R.I., 1966.

\bibitem{newriv:expzeros}
D.~J. Newman and T.~J. Rivlin, \emph{The zeros of the partial sums of the
  exponential function}, J. Approx. Theory \textbf{5} (1972), 405--412.

\bibitem{newriv:expzeroscorrect}
\bysame, \emph{Correction to: ``{T}he zeros of the partial sums of the
  exponential function''}, J. Approx. Theory \textbf{16} (1976), no.~4,
  299--300.

\bibitem{norfolk:widthconj}
T.~S. Norfolk, \emph{Some observations on the {S}aff-{V}arga width conjecture},
  Rocky Mountain J. Math. \textbf{21} (1991), no.~1, 529--538.

\bibitem{norfolk:1f1}
\bysame, \emph{On the zeros of the partial sums to {${}_1F_1(1;b;z)$}}, J.
  Math. Anal. Appl. \textbf{218} (1998), no.~2, 421--438.

\bibitem{zhel:lindelof}
I.~Ostrovskii and N.~Zheltukhina, \emph{The asymptotic zero distribution of
  sections and tails of classical {L}indel\"of functions}, Math. Nachr.
  \textbf{283} (2010), no.~4, 573--587.

\bibitem{rosen:thesis}
P.~C. Rosenbloom, \emph{On sequences of polynomials, especially sections of
  power series}, Ph.D. thesis, Stanford University, 1944, Abstracts in Bull.
  Amer. Math. Soc. \textbf{48} (1942), 839; \textbf{49} (1943), 689.

\bibitem{rosen:distrib}
\bysame, \emph{Distribution of zeros of polynomials}, Lectures on Functions of
  a Complex Variable (W.~Kaplan, ed.), The University of Michigan Press, Ann
  Arbor, 1955, pp.~265--285.

\bibitem{sv:zerofree}
E.~B. Saff and R.~S. Varga, \emph{Zero-free parabolic regions for sequences of
  polynomials}, SIAM J. Math. Anal. \textbf{7} (1976), no.~3, 344--357.

\bibitem{szego:exp}
G.~Szeg\H{o}, \emph{{\"U}ber eine {E}igenschaft der {E}xponentialreihe}, Berlin
  Math. Ges. Sitzungsber. \textbf{23} (1924), 50--64.

\bibitem{vargas:limitcurves}
A.~R. Vargas, \emph{Limit curves for zeros of sections of exponential
  integrals}, Constr. Approx. \textbf{40} (2014), no.~2, 219--239.

\bibitem{zhel:mlsectails}
N.~Zheltukhina, \emph{Asymptotic zero distribution of sections and tails of
  {M}ittag-{L}effler functions}, C. R. Math. Acad. Sci. Paris \textbf{335}
  (2002), no.~2, 133--138.

\end{thebibliography}
%\bibliography{../../../bib/biblio}

\vspace{1cm}

\noindent {\small \textsc{Dept.\ of Mathematics and Statistics \\ Dalhousie University \\ Halifax, Nova Scotia B3H 4J5 \\ Canada}}

\noindent {\small \texttt{antoniov@mathstat.dal.ca}}

\end{document}